\documentclass[12pt,reqno]{amsart}
\usepackage{amssymb,amsmath,amsthm}
\usepackage{graphicx}
\usepackage{subcaption}
\usepackage{fullpage}
\usepackage{scrextend}
\usepackage{siunitx}
\usepackage{enumerate}
\usepackage{tikz,calc}
\usepackage{tikz-cd}
\usetikzlibrary{automata,positioning}
\usepackage{comment}
\usetikzlibrary{calc}
\usepackage{epsfig,graphicx}
\usepackage{listings}
\usepackage{enumerate}
\usepackage{float}
\usepackage{bbm}
\usepackage[colorlinks=true, pdfstartview=FitH, linkcolor=blue, citecolor=blue, urlcolor=blue]{hyperref}

\newcommand{\Mod}[1]{\ (\text{mod}\ #1)}
\renewcommand{\Mod}[1]{{\ifmmode\text{\rm\ (mod~$#1$)}\else\discretionary{}{}{\hbox{ }}\rm(mod~$#1$)\fi}}

\newtheorem{theorem}{Theorem}[section]

\newtheorem{cor}[theorem]{Corollary}

\newtheorem{conjecture}[theorem]{Conjecture}

\theoremstyle{definition} 
\newtheorem{defn}[theorem]{Definition}

\numberwithin{theorem}{section}

\title{Many unit distances requires many directions}

\author{Gabriel Currier}
\address{Department of Mathematics \\ University of British Columbia \\ Vancouver, BC, Canada}
\email{currierg@math.ubc.ca}
 \author{J\'{o}zsef Solymosi}
\address{Department of Mathematics \\ University of British Columbia \\ Vancouver, BC, Canada \\ and Obuda University, Budapest, Hungary}
\email{solymosi@math.ubc.ca}

\begin{document}

\begin{abstract}
    In this note, we show that in planar pointsets determining many unit distances, these unit distances must span many directions. Specifically, we show that a set of $n$ points can determine only $o(n^{4/3})$ unit distances from a set of at most $O(n^{1/3})$ directions.
\end{abstract}

\maketitle

\section{Introduction}

Given a set $P$ of points in $\mathbb{R}^2$, we refer to the number of pairs $(p,p') \in P$ with $||p - p'|| = 1$ as the number of \emph{unit distances determined by $P$}. In $1946$, Erd\H os made the following (now well-known) conjecture, which has become known as the Erd\H{o}s unit distance conjecture \cite{ERD}.

\begin{conjecture}[Erd\H os, 1946]\label{unit}
    Any collection of $n$ points in the plane determines at most $n^{1 +o(1)}$\footnote{See the beginning of section \ref{prelim} for an introduction to big-O notation.} unit distances. 
    
\end{conjecture}

Despite decades of effort, upper and lower bounds on this problem remain quite far apart. The best upper bound, due to Spencer, Szemer\'{e}di and Trotter from $1984$ \cite{SST}, is that the number of unit distances is at most $O(n^{4/3})$. This bound applies to the analogous problem in any norm, and it is tight for some norms; see the discussion below. For the euclidean norm, however, the best-known lower bound is $n^{1 + \frac{c}{\log \log n}}$ unit distances, achieved by an appropriate scaling of a $\sqrt{n} \times \sqrt{n}$ integer grid (see e.g. \cite{MAT}, section $4.2$). This construction dates to the original paper of Erd\H{o}s from 1946 \cite{ERD}.

Despite the lack of progress in the general problem, there have been a number of improvements in specific cases and in closely related problems. One important class of examples are results where we consider only unit distances from some restricted direction set. As an example, Schwartz, Solymosi and de Zeeuw \cite{SSD1} showed that the number of unit distances, where we consider only those coming from rational angles, is at most $n^{1 + 6/\sqrt{\log n}}$. Building on this, Schwartz \cite{SCH} used the subspace theorem to show that the number of unit distances coming from a low rank multiplicative group is at most $n^{1+\epsilon},$ for any $\epsilon >0$ and $n > n_0(\epsilon)$.

Our main theorem here gives another result of this type. Essentially, it says that the number of unit distances from any ``small'' set of directions must be $o(n^{4/3})$. Note that, in contrast to previous results, we make no assumptions on the structure of our direction set, but only on its size.

\begin{theorem}\label{mainthm}
    Let $P$ be a collection of $n$ points, and $D$ a collection of $O(n^{1/3})$ unit vectors. Then, $P$ can determine at most $o(n^{4/3})$ unit distances from vectors in $D$.
\end{theorem}

As mentioned previously, there have long been known examples of norms and sets of $n$ points where $\Omega (n^{4/3})$ unit distances actually can be achieved (see for example \cite{Jarnik1926,VAL}); one of the motivations for theorem \ref{mainthm} is that, until quite recently, all such examples came from a direction set of size $O(n^{1/3})$. Very recently, however, examples \cite{SSz} were found from larger direction sets. More generally, it's known that most norms, in the sense of Baire category, can produce pointsets with only a slightly superlinear number of unit distances. This was shown originally by Matou\v{s}ek \cite{MA}, where he showed that most norms can determine only $O(n \log n \log \log n)$ unit distances. The bound was improved recently by Alon, Buci\'{c} and Sauermann to $O(n \log n)$ in \cite{ABS}.

Our proof will use a variety of tools from incidence geometry and extremal/additive combinatorics. In particular, we will use Guth-Katz polynomial partitioning, the Balog-Szemer\'edi-Gowers theorem, the Freiman-Ruzsa theorem, and a result of Mei-Chu Chang about factorizations of complex numbers coming from generalized arithmetic progressions. There is some overlap of the techniques used here with a previous paper of the authors and White \cite{CSW}. However, the final results of that paper were not as strong, and the proof presented here is simpler and more self-contained. In particular, we avoid using the Szemer\'edi regularity lemma. 

\section{Preliminaries}\label{prelim}

In this section, we will discuss the major tools that we will use in our proof. With the exception of Theorem \ref{MCCfac} and Corollary \ref{circpoints}, these tools are reasonably standard in extremal and additive combinatorics, and the interested reader can skip to Section \ref{mainproof}. To start, we give a brief overview of big-O notation.

\begin{defn}
Let $f(n), g(n)$ be functions of a positive integer $n$. We say that $f(n) = O(g(n))$ if there exists a constant $C$ such that $f(n) \le C g(n)$ for all $n$. We say that $f(n) = \Omega(g(n))$ if $g(n) = O(f(n))$, and that $f(n) = \Theta(g(n))$ if $f(n) = O(g(n))$ and $f(n) = \Omega(g(n))$. Finally, we say that $f(n) = o(g(n))$ if $\lim_{n \to \infty} \frac{f(n)}{g(n)} =0 $. In general, for this paper, we will often allow our constant to depend on other parameters; we will generally include a subscript or specify when this is the case.

\end{defn}

The first tool we will need is the polynomial partitioning theorem of Guth and Katz \cite{GuthKatz15}.

\begin{theorem}\label{polypart}
    Let $P$ be a collection of $n$ points in the plane, and let $d$ be a positive integer. Then, there exists a polynomial $F$ of degree at most $d$ such that there are $O( d^2)$ cells determined by $Z(F)$, and each cell contains  $O(n/d^2)$ points from $P$.
\end{theorem}

Additionally, the following basic result of algebraic geometry, known as B\'ezout's theorem, will help us make full use of Theorem \ref{polypart}

\begin{theorem}\label{bezout}
    Let $F_1,F_2$ be polynomials in two variables of degree $d_1,d_2$, respectively. If $Z(F_1), Z(F_2)$ share no common components, then they have at most $d_1d_2$ points of intersection.
\end{theorem}

Next, we will need the following basic result from extremal graph theory.

\begin{theorem}\label{KST}
    Let $G$ be a graph on $n$ vertices with $ \ge c n^2$ edges, and let $s,t$ be positive integers. Then $G$ contains $\Omega_c( n^{s+t})$ copies of $K_{s,t}$.
\end{theorem}

Finally, we will use some classical results from additive combinatorics. We give a couple of quick definitions to state them.

\begin{defn}
    Let $A$ be a subset of an abelian group, and $H \subset A \times A$. Then, we define the \emph{sum set} $A+_HA$ to be $$A+_HA :=\{a + a' : (a,a') \in H\}$$ and the \emph{difference set}  $A-_HA$ similarly. If the $H$ is omitted, it is assumed that $H = A \times A$.
\end{defn}

Our first tool from additive combinatorics, a version of the Balog-Szemer\'edi-Gowers theorem, says that if some subset of an abelian group generates many ``additive coincidences,'' then it must have a large subset with small difference set. 

\begin{theorem}[Balog-Szemer\' edi-Gowers \cite{BSG1,BSG2}]\label{bsg}
Suppose $G$ is an abelian group, $A \subset G$ is finite, and let $H \subset A \times A$ be such that $H \ge c|A^2|$ and $|A -_H A| \le c'|A|$. Then, there must exist $A' \subset A$ such that $|A'| = \Omega_{c,c'}( |A|)$ and $|A' - A'| = O_{c,c'} (|A'|)$. 
\end{theorem}

The second of these results from additive combinatorics is known as the Freiman-Ruzsa Theorem. In order to state this theorem, we have to introduce the notion of a generalized arithmetic progression.

\begin{defn}
    Let $v_0,\dots,v_d$ be elements of an abelian group, and $\ell_1,\dots,\ell_d$ positive integers. Then, the set $$\{a_0 + \sum_{i=1}^d a_iv_i : 0 \le v_i \le \ell_i -1 \text{ for all } 1 \le i \le d\}$$ is known as a \emph{generalized arithmetic progression} (GAP) of dimension $d$ and size $\prod_i \ell_i$.
\end{defn}

It is not hard to see that a generalized arithmetic progression will generally have small sum and difference sets. The Freiman-Ruzsa theorem  \cite{FREIMAN,RUZSA}, in turn, says that a subset of an abelian group that has small sumset (or difference set) must be contained in a generalized arithmetic progression of bounded size and dimension. The following is the version of the theorem that will be most useful for our purposes.

\begin{theorem}[Freiman-Ruzsa, \cite{FREIMAN,RUZSA}]\label{freiman}
Let $G$ be a torsion-free abelian group, and $A \subset G$ of size $n$. Suppose furthermore that $|A-A| \le cn$. Then, $A$ is contained in a generalized arithmetic progression of dimension $O_c(1)$ and size $O_c( n)$.
\end{theorem}

Our final tool will be a result of Mei-Chu Chang \cite{Chang2003}, showing that, in general, one cannot have too many factorizations of a complex number that come from a generalized arithmetic progression.

\begin{theorem}[Chang, \cite{Chang2003}]\label{MCCfac}
Let $P \subset \mathbb{C}$ be a generalized arithmetic progression of dimension $d$ and size $n$. Then, for a given $\alpha \in \mathbb{C}$, there are at most $e^{O_d(\frac{\log n}{\log \log n})}$ factorizations $\alpha = \alpha_1\alpha_2$ where $\alpha_1,\alpha_2 \in P$. 
\end{theorem}

From this, we can show that a circle can't have too many points from a generalized arithmetic progression.

\begin{cor}\label{circpoints}
Let $P \subset \mathbb{R}^2$ be a generalized arithmetic progression of dimension $d$ and size $n$. Then, a circle $S$ can contain at most $e^{O_d(\frac{\log n}{\log \log n})} = n^{o_d(1)}$ points of $P$.
\end{cor}

\begin{proof}
By performing a translation, we may assume that $S$ has its center at the origin. Henceforth associating $(a,b) \in \mathbb{R}^2$ with $z= a +bi \in \mathbb{C}$, we can write $S$ in the form $z\overline{z} = R$. Thus, we may map each $\alpha \in S \cap P$ to a factorization $\alpha \overline{\alpha} = R$ from $P \cup \overline{P}$. We note two facts: $(1)$ this map is at most two-to-one, in particular, only $\alpha$ and $\overline{\alpha}$ can map to the same factorization, and $(2)$ $P \cup \overline{P}$ is contained in a generalized arithmetic progression of dimension at most $2d+1$ and size at most $2n^2$. Thus, using theorem \ref{MCCfac} to bound the number of such factorizations, we get that $|P \cap S| \le e^{O_d(\frac{\log n}{\log \log n})} = n^{o_d(1)}$
\end{proof}

\section{Proof of the main theorem}\label{mainproof}
In the proof of Theorem \ref{mainthm}, it will be useful to use the language of incidences.

\begin{defn}
    Let $P$ be a collection of points in $\mathbb{R}^2$, and $\mathcal{S}$ be a collection of subsets of $\mathbb{R}^2$. We let $I(P,\mathcal{S})$ be the set of tuples $(p,S) \in P \times \mathcal{S}$ such that $p\in S$. We refer to $|I(P,\mathcal{S})|$ as the number of incidences determined by $P$ and $\mathcal{S}$.
\end{defn}

Before we begin the proof of Theorem \ref{mainthm}, we will give a brief outline. To start, let $P$ be a set of points in the plane determining many unit distances. Using standard tools from incidence geometry, we can show that there must be many collections of points $q_1,q_2,q_3,q_4 \in P$ such that $q_i$ is connected to $q_{i+1}$ by a ``relatively short'' arc from a unit circle. We will think of $q_1,q_2,q_3,q_4$ as simply a path of length $3$. If we also insist that the unit distances determined by $P$ are from a relatively small set of directions, then we can show that these paths of length $3$ must also take a restricted form; moreover, we can show that the existence of each such path forces one of a small number of vector differences between two points of $P$.

The bulk of the proof lies in showing that there are many such paths and that the paths take a sufficiently nice form that we can count them efficiently. When we do this, however, we can then show that there are many copies of a relatively small set of vector differences within $P$. Using theorems \ref{bsg} and \ref{freiman}, this implies that a large portion of $P$ must be contained in a low-dimensional generalized arithmetic progression. However, this cannot happen; since $P$ determines many unit distances, there must be a circle containing many points of $P$, and this, together with Corollary \ref{circpoints}, generates a contradiction.

We can now begin with the proof of Theorem \ref{mainthm}.

\begin{proof}[Proof of Theorem \ref{mainthm}]
Suppose the contrary; that is, that $P$ determines at least $Cn^{4/3}$ unit distances from at most $C'n^{1/3}$ directions, for some large $n$. For the remainder of the proof, we consider only unit distances from these directions; that is, if we say that a collection of points determines a certain number of unit distances, we are counting only those that come from the $C'n^{1/3}$ directions. As such, for our purposes, a given point can determine at most $C'n^{1/3}$ unit distances with other points from $P$. We note, finally, that all constants implicit in this proof are allowed to depend on $C$ and $C'$. 

To start, we iteratively discard from $P$ those points that determine fewer than $(C/2)n^{1/3}$ unit distances. Since we can discard at most $n$ such points, the resulting set (which we will still call $P$) must still determine at least $(C/2)n^{4/3}$ unit distances. Since any individual point can determine at most $C'n^{1/3}$ unit distances with other points from $P$, we know $P$ is still of size at least $\frac{C}{2C'}n$.

Now, we note that there exists a set $S$ of at most $C'n^{1/3}$ points along the unit circle (corresponding to the directions) such that each element of the set $$\mathcal{S} := \{S + p : p \in P\}$$ determines at least $(C/2)n^{1/3}$ incidences with $P$. We let $s_1,\dots,s_m$ be the points of $S$, where $s_1$ is arbitrary, and the remaining indices are determined by the clockwise ordering around the circle. We say that an arc along the unit circle is \emph{good} if it connects two points $s_i$ and $s_j$, where the absolute value of $i-j$ modulo $m$ is at most $C_1$, for some (sufficiently large) constant $C_1 \in \mathbb{Z}^+$ to be specified later. We note that there are at most $C_1C'n^{1/3}$ good arcs.

Now, we apply Theorem \ref{polypart} to the pointset $P$ with $d = (C/32)n^{1/3}$ to get a polynomial $F$ of degree at most $d$ such that $Z(F)$ determines at most $O (n^{2/3})$ cells, where each cell contains $O(n^{1/3})$ points from $P$. We discard, now, any element of $\mathcal{S}$ that's contained in a circle contained entirely within $Z(F)$. $F$ is degree at most $(C/32)n^{1/3}$, and can thus contain at most $(C/64)n^{1/3}$ circles. Thus, after this step we still have $|\mathcal{S}| \ge \frac{C}{2C'}n - \frac{C}{64}n^{1/3} \ge \frac{C}{4C'}n$.

Now, let $a$ be a good arc, and $e = p+a$ for some $p \in P$. We say that $e$ is a \emph{good edge} if 

\begin{enumerate}
    \item the endpoints of $e$ are points $p_1,p_2 \in P$
    \item $e$ contains no points of $(p+S) \cap P$ except $p_1,p_2$.
    \item $e$ is contained entirely on the interior of a cell determined by $Z(F)$
\end{enumerate}
Our first claim says that there are many such edges.
\\\\
{\bfseries Claim 1:} There are at least $\frac{C^2}{32C'}n^{4/3}$ good edges.

\begin{proof}Consider an arbitrary element $p + S \in \mathcal{S}.$ We know that $p+S$ is contained in a unit circle that is not contained entirely within $Z(F)$; therefore, by Theorem \ref{bezout} we know that this circle intersects $Z(F)$ in at most $(C/16)n^{1/3}$ points. As noted before, each element of $\mathcal{S}$ determines at least $(C/2)n^{1/3}$ incidences with $P$. Let $i_1,\dots,i_k \subset [m]$ be the set of indices (ordered smallest to largest) such that $p+s_{i_j} \in P$ for all $1 \le j \le k$, and note $k \ge (C/2)n^{1/3}$. At most $(C'/C_1)n^{1/3}$ of these indices $i_j$ can be such that $i_{j+1} - i_j$ (where for $j = k$, we take $i_1- i_k$ modulo $m$) is greater than $C_1$ . Thus, letting $C_1 = (4C'/C)$, we conclude that there are at least $(C/4)n^{1/3}$ indices $i_j$ such that $p+s_{i_j}$ and $p+s_{i_{j+1}}$ are connected by translates of good arcs. Finally, we note that, by theorem \ref{bezout}, at most $(C/8)n^{1/3}$ of these good arcs can have an intersection with $Z(F)$. Thus, for every element of $\mathcal{S}$, there exist at least $(C/8)n^{1/3}$ of these pairs where their arcs are entirely in-cell and are thus connected by good edges. Since $|\mathcal{S}| \ge \frac{C}{4C'}n$, there are a total of at least $\frac{C^2}{32C'}n^{4/3}$ good edges.

\end{proof}

Now, let $C_2 >0$ be a small constant to be chosen later. We refer to a cell determined by $Z(F)$ as good if there are at least $C_2n^{2/3}$ good edges contained in its interior.
\\\\
{\bfseries Claim 2:} At least $\Omega (n^{2/3})$ cells determined by $Z(F)$ are good.
\begin{proof}
We discard cells that contain fewer than  $C_2n^{2/3}$ good edges. Note that since there are $O (n^{2/3})$ cells, we lose $O (C_2n^{4/3})$ good edges in this process. Letting $C_2$ be sufficiently small and using Claim $1$, we will still have $\frac{C^2}{64C'}n^{4/3}$ good edges at the end of this process. A given cell can contain $O (n^{1/3})$ points and thus $O (n^{2/3})$ good edges. Therefore, there must be $\Omega (n^{2/3})$ good cells.
\end{proof}
We define an \emph{in-cell self-intersecting $P_3$} to be a path of length $3$ consisting of three good edges contained entirely within a given cell (see Figure \ref{p_3_fig}).
\\\\{\bfseries Claim 3:} In a good cell there are $\Omega (n^{4/3})$ in-cell self-intersecting $P_3$'s.

\begin{figure}

    \centering
    \includegraphics[scale=.5]{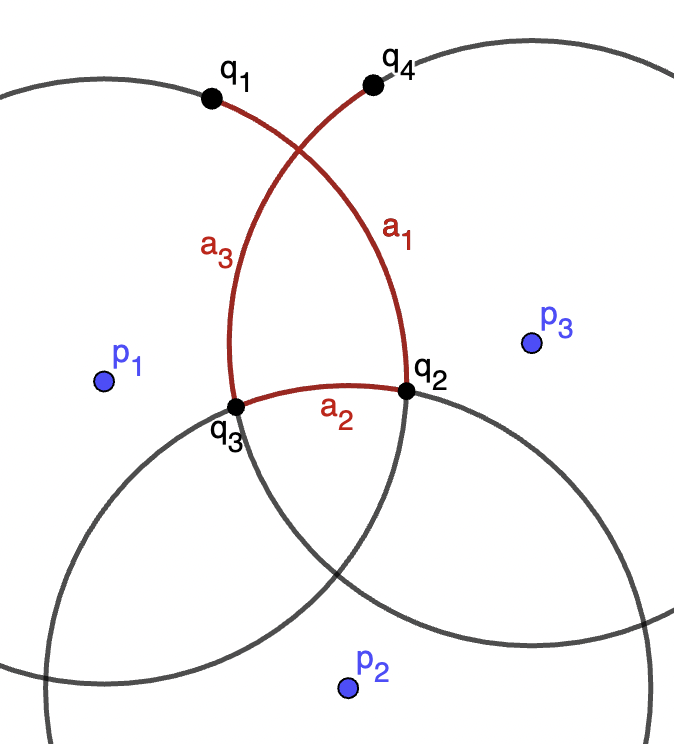}
    \caption{A self-intersecting $P_3$ consisting of points $q_1,q_2,q_3,q_4$. Note that the vector difference $p_1-p_3$ is uniquely determined by the choice of arcs $a_1,a_2,a_3$.}
    \label{p_3_fig}
\end{figure}

\begin{proof}
Consider a good cell $D$ with in-cell points $P_D$. We note $|P_D| = O (n^{1/3})$ by theorem \ref{polypart}. Since $|P_D|$ also contains at least  $C_2n^{2/3} = \Omega (n^{2/3})$ good edges, and each pair of points can be connected by at most $2$ good edges, we must in fact have $|P_D| = \Theta (n^{1/3})$. Now, we let $G_D$ be the graph obtained by connecting points of $P_D$ that are connected by good edges. As noted before, a given pair of points is connected by at most two good edges, so this graph must have $\Omega (n^{2/3})$ edges. By Theorem \ref{KST}, it must contain $\Omega (n^2)$ copies of $K_{3,3}$. Since $K_{3,3}$ is non-planar, any copy of $K_{3,3}$ in $G_D$, when considered as a geometric graph with edges given by the good edges connecting its points, must contain two crossing edges $e_1 = (q_1,q_2)$ and $e_3 = (q_3,q_4)$. Since $K_{3,3}$ is complete bipartite, $q_2$ and $q_3$ must be connected by a good edge $e_2$; thus, any copy of $K_{3,3}$ contains a self-intersecting in-cell $P_3$ as a subgraph. Furthermore, for any self-intersecting in-cell $P_3$ consisting of points $q_1,q_2,q_3,q_4$, there are at most $O(n^{2/3})$ pairs of points in $P_D \setminus \{q_1,q_2,q_3,q_4\}$ with which we could construct a copy of $K_{3,3}$. Thus, any self-intersecting $P_3$ can be counted in this way at most $O (n^{2/3})$ times. We conclude there are at least $\Omega (n^{4/3})$ in-cell self-intersecting $P_3$'s in $D$.
\end{proof}

By claims $2$ and $3$, there are $\Omega (n^{2/3})$ good cells and each contains $\Omega(n^{4/3})$ self-intersecting $P_3$'s. Thus, there are a total of $\Omega (n^2)$ in-cell self-intersecting $P_3$'s. 

Now, each in-cell self-intersecting $P_3$ gives points $p_1,p_2,p_3,q_1,q_2,q_3,q_4 \in P$ such that $q_1,q_2,q_3,q_4$ are connected by an in-cell self-intersecting $P_3$ consisting of translates of good arcs $a_1,a_2,a_3$ lying on unit circles centered at $p_1,p_2,p_3$ (respectively). See Figure \ref{p_3_fig} for reference. Let $H \subset P \times P$ consist of pairs of points $(p_1,p_3)$ for which there exist such a self-intersecting $P_3$; we claim $|H| = \Omega(n^2)$. Since there are $\Omega(n^2)$ self-intersecting $P_3$'s, we need only show that each pair $p_1,p_3$ is counted in this way at most $O(1)$ times.

Fix such a pair $p_1,p_3$. We observe first that the unit circles centered at $p_1$ and $p_3$ can intersect at most two points. If an intersection point is on the boundary of $Z(F)$, then no in-cell self-intersecting $P_3$ can come from this intersection. If an intersection is in-cell, then by property (b) of good edges, there is a unique choice of points $q_1,q_2 \in (p_1 + S) \cap P$ and $q_3,q_4 \in (p_2 + S) \cap P$ that can form good edges containing the intersection point. Finally, there will be most $8$ choices of $p_2$. Thus, $p_1,p_3$ can be counted in at most $16$ self-intersecting $P_3$'s.

To finish the proof, note that the vector difference $p_1 - p_3$ is uniquely determined by the choice of good arcs $a_1,a_2,a_3$ determining the self-intersecting $P_3$. Since, as noted before, there are only $C_1C'n^{1/3}$ good arcs, there $O(n)$ possible such vector differences. Thus, $|P-_H P| = O(n)$. Using this, we can apply Theorem \ref{bsg} to conclude there exists a set $P' \subset P$ such that $|P'| = \Omega (n)$ and $|P'-P'| = O (n)$. By Theorem \ref{freiman}, we observe that $P'$ must be contained in a generalized arithmetic progression of dimension $O(1)$ and size $O(n)$.

Recall that each point in $P$ determines at least $(C/2)n^{1/3}$ unit distances with other points of $P$; thus, the points of $P'$ must determine $\Omega(n^{4/3})$ unit distances with other points of $P$. By a simple double-counting argument, a point in $P$ must exist that determines $\Omega(n^{1/3})$ unit distances with points of $P'$. Equivalently, a unit circle must contain $\Omega(n^{1/3})$ points from $P'$. However, this is a contradiction by theorem \ref{circpoints} since $P'$ is contained in a generalized arithmetic progression of dimension $O(1)$ and size $O(n)$. This completes the proof.
\end{proof}

\section{Concluding Remarks} We showed that assuming many unit distances parallel to a small set of directions forces a strong arithmetic structure; that is, it guarantees many points are from a low dimensional generalized arithmetic progression. Then, according to Chang's result, we conclude that such pointsets have a small intersection with any circle. It is possible that one can always find strong arithmetic structure in any point set with many unit distances. 
Under the stronger assumption of few distinct distances, Erd\H{o}s conjectured that a planar pointset ''\textit{has lattice structure}" \cite{ERD2}. If true, and lattice structure meant generalized arithmetic progression, then it would imply an alternative proof for the result of Guth and Katz on distinct distances \cite{GuthKatz15}. The strongest result related to Erd\H{o}s' question is in \cite{LuSeZe}, but no better structural result is known. 

\section{Acknowledgements}

The authors thank Sean Eberhard for suggesting using Chang's result, and to Greg Martin for helpful discussions. The second-named author's research was supported by an NSERC Discovery grant and OTKA K grant no. 133819.

\bibliographystyle{amsplain}
\bibliography{bibliography}

\end{document}